\documentclass{amsart}
\usepackage{amssymb}
\usepackage{amsmath}
\usepackage{amsfonts}

\newtheorem{theorem}{Theorem}
\theoremstyle{plain}

\newtheorem{definition}{Definition}

\newtheorem{remark}{Remark}

\numberwithin{equation}{section}
\input{tcilatex}

\begin{document}
\title[Ergodic Theorem in Grand Variable Exponent Lebesgue Spaces]{Ergodic
Theorem in Grand Variable Exponent Lebesgue Spaces}
\author{Cihan UNAL}
\address{Assessment, Selection and Placement Center, Ankara, TURKEY}
\email{cihanunal88@gmail.com}
\urladdr{}
\thanks{}
\date{}
\dedicatory{}
\thanks{}

\begin{abstract}
We consider several fundamental properties of grand variable exponent
Lebesgue spaces. Moreover, we discuss Ergodic theorems in these spaces
whenever the exponent is invariant under the transformation.
\end{abstract}

\maketitle

\section{Introduction}

In 1992, Iwaniec and Sbordone \cite{Iw} introduced grand Lebesgue spaces $%
L^{p)}\left( \Omega \right) $, ($1<p<\infty $), on bounded sets $\Omega
\subset 
\mathbb{R}
^{d}$ with applications to differential equations. A generalized version $%
L^{p),\theta }\left( \Omega \right) $ appeared in Greco et al. \cite{Gr}.
These spaces has been intensively investigated recently due to several
applications, see \cite{Ca}, \cite{Di}, \cite{Fi}, \cite{Fior}, \cite{Ko}, 
\cite{Ra}. Also the solutions of some nonlinear differential equations were
studied in these spaces, see \cite{Fs}, \cite{Gr}. The variable exponent
Lebesgue spaces (or generalized Lebesgue spaces) $L^{p(.)}$ appeared in
literature for the first time in 1931 with an article written by Orlicz \cite%
{Orli}. Kov\'{a}\v{c}ik and R\'{a}kosn\'{\i}k \cite{Kor} introduced the
variable exponent Lebesgue space $L^{p(.)}(%
\mathbb{R}
^{d})$ and Sobolev space $W^{k,p(.)}(%
\mathbb{R}
^{d})$ in higher dimensions Euclidean spaces. The spaces $L^{p(.)}(%
\mathbb{R}
^{d})$ and $L^{p}(%
\mathbb{R}
^{d})$ have many common properties such as Banach space, reflexivity,
separability, uniform convexity, H\"{o}lder inequalities and embeddings. A
crucial difference between $L^{p(.)}(%
\mathbb{R}
^{d})$ and $L^{p}(%
\mathbb{R}
^{d})$ is that the variable exponent Lebesgue space is not invariant under
translation in general, see \cite[Lemma 2.3]{Die} and \cite[Example 2.9]{Kor}%
. For more information, we refer \cite{Cr}, \cite{Dien} and \cite{Fan}.
Moreover, the space $L^{p(.)}\left( \Omega \right) $ was studied by \cite%
{Aoy}, where $\Omega $ is a probability space. The grand variable exponent
Lebesgue space $L^{p(.),\theta }\left( \Omega \right) $ was introduced and
studied by Kokilashvili and Meskhi \cite{Ko}. In this work, they established
the boundedness of maximal and Calderon operators in these spaces. Moreover,
the space $L^{p(.),\theta }\left( \Omega \right) $ is not reflexive,
separable, rearrangement invariant and translation invariant.

In this study, we give some basic properties of $L^{p(.),\theta }\left(
\Omega \right) ,$ and consider Birkhoff's Ergodic Theorem in the context of
a certain subspace of the grand variable exponent Lebesgue space $%
L^{p(.),\theta }\left( \Omega \right) $. So, we have more general results in
sense to Gorka \cite{Gor} in these spaces.

\section{\textbf{Notations and Preliminaries}}

\begin{definition}
Assume that $\left( \Omega ,\Sigma ,\mu \right) $ is a probability space and
let $p\left( .\right) :\Omega \longrightarrow \left[ 1,\infty \right) $ be a
measurable function (variable exponent) such that 
\begin{equation*}
1\leq p^{-}=\underset{x\in \Omega }{\text{essinf}}p\left( x\right) \leq 
\underset{x\in \Omega }{\text{esssup}}p\left( x\right) =p^{+}<\infty .
\end{equation*}%
The variable exponent Lebesgue space $L^{p(.)}(\Omega )$ is defined as the
set of all measurable functions $f$ on $\Omega $ such that $\varrho
_{p(.)}(\lambda f)<\infty $ for some $\lambda >0$, equipped with the
Luxemburg norm%
\begin{equation*}
\left\Vert f\right\Vert _{p(.)}=\inf \left\{ \lambda >0:\varrho _{p\left(
.\right) }\left( \frac{f}{\lambda }\right) \leq 1\right\} \text{,}
\end{equation*}%
where $\varrho _{p(.)}(f)=\dint\limits_{\Omega }\left\vert f(x)\right\vert
^{p(x)}d\mu \left( x\right) .$ The space $L^{p(.)}(\Omega )$ is a Banach
space with respect to $\left\Vert .\right\Vert _{p(.)}$. Moreover, the norm $%
\left\Vert .\right\Vert _{p(.)}$ coincides with the usual Lebesgue norm $%
\left\Vert .\right\Vert _{p}$ whenever $p(.)=p$ is a constant function. Let $%
p^{+}<\infty $. Then $f\in L^{p(.)}(\Omega )$ if and only if $\varrho
_{p(.)}(f)<\infty $.
\end{definition}

\begin{definition}
Let $\theta >0.$ The grand variable exponent Lebesgue spaces $L^{p(.),\theta
}\left( \Omega \right) $ is the class of all measurable functions for which%
\begin{equation*}
\left\Vert f\right\Vert _{p(.),\theta }=\sup_{0<\varepsilon
<p^{-}-1}\varepsilon ^{\frac{\theta }{p^{-}-\varepsilon }}\left\Vert
f\right\Vert _{p(.)-\varepsilon }<\infty .
\end{equation*}%
When $p(.)=p$ is a constant function, these spaces coincide with the grand
Lebesgue spaces $L^{p),\theta }\left( \Omega \right) $.
\end{definition}

It is easy to see that we have%
\begin{equation}
L^{p(.)}\hookrightarrow L^{p(.),\theta }\hookrightarrow L^{p(.)-\varepsilon
}\hookrightarrow L^{1}\text{, }0<\varepsilon <p^{-}-1  \label{2.1}
\end{equation}%
due to $\left\vert \Omega \right\vert <\infty ,$ see \cite{Da}, \cite{Ko}, 
\cite{Ra}.

\begin{remark}
Let $C_{0}^{\infty }(\Omega )$ be the space of smooth functions with compact
support in $\Omega .$ It is well known that $C_{0}^{\infty }(\Omega )$ is
not dense in $L^{p(.),\theta }\left( \Omega \right) $, i.e., the closure of $%
C_{0}^{\infty }(\Omega )$ with respect to the $\left\Vert .\right\Vert
_{p(.),\theta }$ norm does not coincide with the space $L^{p(.),\theta
}\left( \Omega \right) $. Now, we denote $\left[ L^{p(.)}\left( \Omega
\right) \right] _{p(.),\theta }$ as the closure of $C_{0}^{\infty }(\Omega )$
in $L^{p(.),\theta }\left( \Omega \right) $. Hence this closure is obtained
as%
\begin{equation*}
\left\{ f\in L^{p(.),\theta }\left( \Omega \right) :\lim_{\varepsilon
\rightarrow 0}\varepsilon ^{\frac{\theta }{p^{-}-\varepsilon }}\left\Vert
f\right\Vert _{p(.)-\varepsilon ,w}=0\right\}
\end{equation*}%
, see \cite{Da}, \cite{Gr}, \cite{Ko}. Moreover, we have%
\begin{equation*}
C_{0}^{\infty }(\Omega )\subset L^{p(.)}\left( \Omega \right) \subset \left[
L^{p(.)}\left( \Omega \right) \right] _{p(.),\theta }\text{ and }\left[
L^{p(.)}\left( \Omega \right) \right] _{p(.),\theta }=\overline{%
C_{0}^{\infty }(\Omega )}.
\end{equation*}
\end{remark}

\begin{definition}
Let $\left( G,\Sigma ,\mu \right) $ be a measure space. A measurable
function $T:G\longrightarrow G$ is called a measure-preserving
transformation if%
\begin{equation*}
\mu \left( T^{-1}(A)\right) =\mu \left( A\right)
\end{equation*}%
for all $A\in \Sigma .$
\end{definition}

\section{Main Results}

In the following theorem, we obtain more general result then \cite[Theorem
3.1]{Gor} since $L^{p(.)}\left( \Omega \right) \subset \left[ L^{p(.)}\left(
\Omega \right) \right] _{p(.),\theta }\subset L^{p(.),\theta }\left( \Omega
\right) $.

\begin{theorem}
Let $\left( \Omega ,\Sigma ,\mu \right) $ be a probability space and $%
T:\Omega \longrightarrow \Omega $ a measure preserving transformation.
Moreover, if $p(.)$ is $T$-invariant, i.e., $p(T(.))=p(.)$, then

(i) The limit 
\begin{equation*}
f_{av}(x)=\lim_{n\rightarrow \infty }\frac{1}{n}\tsum\limits_{j=0}^{n-1}f%
\left( T^{j}(x)\right)
\end{equation*}%
exists for all $f\in L^{p(.),\theta }\left( \Omega \right) $ and almost each
point $x\in \Omega $, and $f_{av}\in L^{p(.),\theta }\left( \Omega \right) $.

(ii) For every $f\in L^{p(.),\theta }\left( \Omega \right) $, we have 
\begin{equation}
f_{av}(x)=f_{av}\left( T(x)\right) ,  \label{3.1}
\end{equation}%
\begin{equation}
\dint\limits_{\Omega }f_{av}d\mu =\dint\limits_{\Omega }fd\mu .  \label{3.2}
\end{equation}

(iii) For all $f\in \left[ L^{p(.)}\left( \Omega \right) \right]
_{p(.),\theta }$, we get 
\begin{equation}
\lim_{n\rightarrow \infty }\left\Vert f_{av}-\frac{1}{n}\tsum%
\limits_{j=0}^{n-1}f\circ T^{j}\right\Vert _{p(.),\theta }=0.  \label{3.3}
\end{equation}
\end{theorem}

\begin{proof}
By (\ref{2.1}), the existence of the limit $f_{av}(x)$ for almost every
point in $\Omega $ follows from the standard Birkhoof's Theorem. By Fatou's
Lemma and the definition of the norm $\left\Vert .\right\Vert _{p(.),\theta
} $, we have%
\begin{eqnarray*}
\dint\limits_{\Omega }\left\vert f_{av}(x)\right\vert ^{p(x)-\varepsilon
}d\mu &=&\dint\limits_{\Omega }\left\vert \lim_{n\rightarrow \infty }\frac{1%
}{n}\tsum\limits_{j=0}^{n-1}f\left( T^{j}(x)\right) \right\vert
^{p(x)-\varepsilon }d\mu \\
&\leq &\dint\limits_{\Omega }\lim_{n\rightarrow \infty }\left( \frac{1}{n}%
\tsum\limits_{j=0}^{n-1}\left\vert f\left( T^{j}(x)\right) \right\vert
\right) ^{p(x)-\varepsilon }d\mu \\
&\leq &\liminf_{n\rightarrow \infty }\dint\limits_{\Omega }\left( \frac{1}{n}%
\tsum\limits_{j=0}^{n-1}\left\vert f\left( T^{j}(x)\right) \right\vert
\right) ^{p(x)-\varepsilon }d\mu \\
&\leq &\liminf_{n\rightarrow \infty }\frac{1}{n}\tsum\limits_{j=0}^{n-1}%
\dint\limits_{\Omega }\left\vert f\left( T^{j}(x)\right) \right\vert
^{p(x)-\varepsilon }d\mu
\end{eqnarray*}%
for any $\varepsilon \in \left( 0,p^{-}-1\right) .$ Here, we used convexity
and Jensen inequality in last step. Moreover, since $T$ is a measure
preserving map and $p(.)$ is $T$-invariant, we get 
\begin{equation*}
\dint\limits_{\Omega }\left\vert f(T(x))\right\vert ^{p(x)-\varepsilon }d\mu
=\dint\limits_{\Omega }\left\vert f(T(x))\right\vert ^{p(T(x))-\varepsilon
}d\mu =\dint\limits_{\Omega }\left\vert f(x)\right\vert ^{p(x)-\varepsilon
}d\mu .
\end{equation*}%
This follows that%
\begin{equation}
\dint\limits_{\Omega }\left\vert f_{av}(x)\right\vert ^{p(x)-\varepsilon
}d\mu \leq \dint\limits_{\Omega }\left\vert f(x)\right\vert
^{p(x)-\varepsilon }d\mu <\infty .  \label{3.4}
\end{equation}%
Thus, we obtain%
\begin{eqnarray*}
\left\Vert f_{av}\right\Vert _{p(.),\theta } &=&\sup_{0<\varepsilon
<p^{-}-1}\varepsilon ^{\frac{\theta }{p^{-}-\varepsilon }}\left\Vert
f_{av}\right\Vert _{p(.)-\varepsilon } \\
&\leq &\sup_{0<\varepsilon <p^{-}-1}\varepsilon ^{\frac{\theta }{%
p^{-}-\varepsilon }}\left\Vert f\right\Vert _{p(.)-\varepsilon }<\infty
\end{eqnarray*}%
and $f_{av}\in L^{p(.),\theta }\left( \Omega \right) .$ This completes 
\textit{(i)}. By the Ergodic Theorem in classical Lebesgue spaces, we have (%
\ref{3.1}) and (\ref{3.2}) immediately. In order to prove (\ref{3.3}), we
assume that $f\in C_{0}^{\infty }(\Omega )$. Thus, $f\in L^{\infty }(\Omega
) $ and 
\begin{eqnarray*}
\lim_{n\rightarrow \infty }\left\vert f_{av}(x)-\frac{1}{n}%
\tsum\limits_{j=0}^{n-1}f\left( T^{j}(x)\right) \right\vert
^{p(x)-\varepsilon } &=&0\text{, a.e.} \\
\left\Vert f_{av}\right\Vert _{L^{\infty }(\Omega )} &\leq &\left\Vert
f\right\Vert _{L^{\infty }(\Omega )}
\end{eqnarray*}%
for any $\varepsilon \in \left( 0,p^{-}-1\right) .$ Therefore, we have%
\begin{eqnarray*}
\left\vert f_{av}(x)-\frac{1}{n}\tsum\limits_{j=0}^{n-1}f\left(
T^{j}(x)\right) \right\vert ^{p(x)-\varepsilon } &\leq &\left\vert
\left\Vert f\right\Vert _{L^{\infty }(\Omega )}+\frac{1}{n}%
\tsum\limits_{j=0}^{n-1}\left\Vert f\left( T^{j}\right) \right\Vert
_{L^{\infty }(\Omega )}\right\vert ^{p(x)-\varepsilon } \\
&\leq &2^{p^{+}}\left( \left\Vert f\right\Vert _{L^{\infty }(G)}+1\right)
^{p^{+}-\varepsilon }\in L^{1}(\Omega ).
\end{eqnarray*}%
Hence, by Lebesgue dominated convergence theorem, we have (\ref{3.3}) and
provided $f\in C_{0}^{\infty }(\Omega ).$ Since $C_{0}^{\infty }(\Omega )$
is dense in $\left[ L^{p(.)}\left( \Omega \right) \right] _{p(.),\theta }$
with respect to the norm $\left\Vert .\right\Vert _{p(.),\theta }$, for any $%
f\in \left[ L^{p(.)}\left( \Omega \right) \right] _{p(.),\theta }$ and $\eta
>0$ there is a $g\in C_{0}^{\infty }(\Omega )$ such that%
\begin{equation}
\left\Vert f-g\right\Vert _{p(.),\theta }<\eta .  \label{3.5}
\end{equation}%
By the previous step, there is an $n_{0}$ such that%
\begin{equation}
\left\Vert g_{av}-\frac{1}{n}\tsum\limits_{j=0}^{n-1}g\circ T^{j}\right\Vert
_{p(.)-\varepsilon }<\eta  \label{3.6}
\end{equation}%
for $n\geq n_{0}$ and $\varepsilon \in \left( 0,p^{-}-1\right) $. Hence, we
have 
\begin{equation}
\left\Vert g_{av}-\frac{1}{n}\tsum\limits_{j=0}^{n-1}g\circ T^{j}\right\Vert
_{p(.),\theta }<\eta  \label{3.7}
\end{equation}%
by (\ref{3.6}) and the definition of the norm $\left\Vert .\right\Vert
_{p(.),\theta }$. This follows from (\ref{3.4}), (\ref{3.5}) and (\ref{3.7})
that%
\begin{eqnarray*}
\left\Vert f_{av}-\frac{1}{n}\tsum\limits_{j=0}^{n-1}f\circ T^{j}\right\Vert
_{p(.),\theta } &\leq &\left\Vert f_{av}-g_{av}\right\Vert _{p(.),\theta
}+\left\Vert g_{av}-\frac{1}{n}\tsum\limits_{j=0}^{n-1}g\circ
T^{j}\right\Vert _{p(.),\theta } \\
&&+\left\Vert \frac{1}{n}\tsum\limits_{j=0}^{n-1}\left( f-g\right) \circ
T^{j}\right\Vert _{p(.),\theta } \\
&\leq &2\left\Vert f-g\right\Vert _{p(.),\theta }+\left\Vert g_{av}-\frac{1}{%
n}\tsum\limits_{j=0}^{n-1}g\circ T^{j}\right\Vert _{p(.),\theta } \\
&<&\frac{\eta }{2}+\frac{\eta }{2}=\eta .
\end{eqnarray*}%
That is the desired result.
\end{proof}

\end{document}